%% file: Genevois_Martin_Graph_Products_Cycle.tex
\documentclass[ 11pt]{article}

\usepackage[dvipsnames]{xcolor}
\usepackage[utf8]{inputenc}  
\usepackage[T1]{fontenc}         
\usepackage{geometry}         
\usepackage[english]{babel}  
\usepackage{cite}
\usepackage{amssymb,  xfrac, stmaryrd, blkarray, multirow, enumerate,
  verbatim,xifthen}   
\usepackage{float}
\usepackage{amsmath} 
\usepackage{amscd}
\usepackage{graphicx, color, pinlabel}    
\usepackage{amsthm}      

\usepackage{mathrsfs}
           
\usepackage{latexsym}     
\usepackage[all]{xy}
\usepackage{hyperref}
\newtheorem{thm}{Theorem}[section]

\theoremstyle{theorem}

\theoremstyle{definition}
\newtheorem{definition}[thm]{Definition}

\newtheorem{rmk}[thm]{Remark}

\newtheorem*{thm*}{Theorem}

\newtheorem{obs}[thm]{Observation}

\newtheorem*{definition*}{Definition}

\newtheorem*{cor*}{Corollary}
\newtheorem{prop}[thm]{Proposition}
\newtheorem{cor}[thm]{Corollary}
\newtheorem*{prop*}{Proposition}
\newtheorem{lem}[thm]{Lemma}
\newtheorem*{claim*}{Claim}

\newtheorem*{thmA}{Theorem A}
\newtheorem*{corB}{Corollary B}
\newtheorem*{corC}{Corollary C}
\newtheorem*{thmD}{Theorem D}
\newtheorem*{thmE}{Theorem E}

\def\lquotient#1#2{%
\makeatletter
\lower.6ex\hbox{$#1$}\backslash\raise.3ex\hbox{$#2$}%
\makeatother
}

\def\rquotient#1#2{%
\makeatletter
\raise.6ex\hbox{$#1$}/\lower.2ex\hbox{$#2$}%
\makeatother
}

\newcommand{\bbQ}{{\mathbb Q}}

\newcommand{\bbZ}{{\mathbb Z}}

\newcommand{\cG}{{\mathcal G}}

\newcommand{\ra}{\rightarrow}

\def\rquotient#1#2{%
\makeatletter
\raise.3ex\hbox{$#1$}/\lower.3ex\hbox{$#2$}%
\makeatother
}	

\makeatletter
\newcommand{\subjclass}[2][2010]{%
  \let\@oldtitle\@title%
  \gdef\@title{\@oldtitle\footnotetext{#1 \emph{Mathematics subject classification.} #2}}%
}
\newcommand{\keywords}[1]{%
  \let\@@oldtitle\@title%
  \gdef\@title{\@@oldtitle\footnotetext{\emph{Key words and phrases.} #1.}}%
}
\makeatother

\newcommand{\Address}{{% additional braces for segregating \footnotesize
  \bigskip
  \small

   \textsc{Department of Mathematics, Faculty of Sciences, University Aix-Marseille, 3 place Victor Hugo, 13331 Marseille cedex 3, France.}\par\nopagebreak
  \textit{E-mail address}: \texttt{anthony.genevois@univ-amu.fr}

  \medskip
  
   \textsc{Department of Mathematics and the Maxwell Institute for Mathematical Sciences, Heriot-Watt University, Riccarton, EH14 4AS Edinburgh, United Kingdom.}\par\nopagebreak
  \textit{E-mail address}: \texttt{alexandre.martin@hw.ac.uk}
 
}}

\makeatletter
\newdimen\bibindent
\makeatother

\title{\textbf{Automorphism groups of cyclic products of groups}}  
\author{Anthony Genevois and Alexandre Martin}
\date{}

\subjclass{ Primary 20F65. Secondary 20F28.}
\keywords{graph products of groups, CAT(0) cube complexes, acylindrical hyperbolicity, automorphism group}

\begin{document}
\maketitle

 \begin{abstract} 
This article initiates a geometric study of the automorphism groups of general graph products of groups, and investigates the algebraic and geometric structure of automorphism groups of \textit{cyclic} product of groups. For a cyclic product of at least five groups, we show that the action of the cyclic product on its Davis complex extends to an action of the whole automorphism group. This action allows us to completely compute the automorphism group and to derive several of its properties: Tits Alternative, acylindrical hyperbolicity, lack of property (T).
 \end{abstract}

\tableofcontents

\input{Intro}

\input{Preliminaries}
\input{Walls}

\input{AlgCharact}

\input{Automorphisms}

\input{Acyl}

\addcontentsline{toc}{section}{References}

\bibliographystyle{alpha}
\bibliography{Genevois_Martin_Graph_Products_Cycle}

\Address

\end{document}

%% file: Intro.tex
\section*{Introduction}

Given a group, a natural question is to determine its (outer) automorphism group. Only few automorphism groups have been studied from a geometric point of view. Indeed, while many groups come with interesting actions associated to them, there is no general recipe for constructing a `nice' action of $\mathrm{Aut}(G)$ out of an action of $G$. Famous examples where automorphism groups have been studied from a geometric perspective  include the (outer) automorphism groups of free groups, which act on their outer-spaces and various hyperbolic graphs (see \cite{VogtmannSurveyOut, VogtmannSurveyOutRecent}); and the (outer) automorphism groups of surface groups, which essentially coincide with the corresponding mapping class groups and thus act on their Teichm\"uller spaces and their curve complexes (see \cite{IvanovSurveyMCG}).

 An interesting case where one can study the automorphism group from a geometric perspective is when the original action satisfies some form of `rigidity', i.e. when  the action of a group $G$ on a space $X$ can be extended to an action of $\mathrm{Aut}(G)$ on $X$ (where one identifies a centreless group $G$ with the subgroup of $\mathrm{Aut}(G)$ consisting of its inner automorphisms). The prime example of this phenomenon is the work of Ivanov on the action of mapping class groups of hyperbolic surfaces on their curve complexes: Ivanov showed that an automorphism of the mapping class group induces an automorphism of the underlying curve complex. Another example is given by the Higman group: in \cite{MartinHigman}, the second author computed the automorphism group of the Higman group $H_4$  by first extending the action of $H_4$ on a CAT(0) square complex naturally associated to its standard presentation to an action of $\mathrm{Aut}(H_4)$. Such a rigidity phenomenon thus provides a fruitful road towards understanding automorphism groups, and the goal of this article is to initiate such a geometric approach for the study of the automorphism groups of graph products of groups.

Graph products of groups, which have been introduced by  Green in \cite{GraphProduct}, are a class of group products that, loosely speaking, interpolates between free and direct products (see Section \ref{section:graphproducts} for a precise definition). They include two intensively studied families of groups: right-angled Artin groups  and right-angled Coxeter groups. Many articles have been dedicated to the study of the automorphism groups of these particular examples of graph products. Beyond that, most of the literature on the automorphisms of other types of graph products has focused on free products \cite{OutSpaceFreeProduct, HorbezHypGraphsForFreeProducts, HorbezTitsAlt} and graph products of \textit{abelian} groups \cite{AutGPabelianSet, AutGPabelian, AutGPSIL, RuaneWitzel}. By contrast,  automorphism groups of graph products of more general groups are essentially uncharted territory: for instance, no set of generators is known in general.

%In this context, we should mention ???, ???, ???. 
%Nevertheless, no convincing geometric framework is known for the study of automorphisms of graph products, even for right-angled Artin or Coxeter groups. (Here, we should mention the interesting attempt made in ??? to define an outer-space for right-angled Artin groups. 
Our goal is to study automorphism groups of graph products from the point of view of interesting actions on non-positively curved spaces. As mentioned above, automorphism groups of free groups, or more generally various right-angled Artin groups \cite{AutTwoDimRaag, OutRaagUntwisted}, already possess  interesting actions on variations of outer spaces, and such actions have been studied with great success.  However, we emphasise that the philosophy of this article is of a quite different nature:  articles on $\mathrm{Out(RAAGs)}$ generally use the structure of the underlying right-angled Artin group to construct a proper action of the automorphism group. By contrast, we do not assume here any prior knowledge of the groups constituting the graph product, and we want to find  an action for its automorphism group that will actually reveal its algebraic structure (generators, decomposition, etc.), together with other interesting properties. \\

In this article, we illustrate our geometric approach in the specific  case of \emph{cyclic products} of groups, ie., graph products of non-trivial groups over a cycle of length at least five.  To the authors' knowledge, the results represent the first results on the algebraic and geometric structure of automorphism groups of graph product of general (and in particular non-abelian) groups. 

  First of all, we completely describe  the automorphisms of such graph products:

\begin{thmA}
For every $n \geq 5$ and every collection of non trivial groups $\mathcal{G}= \{G_i, i\in \bbZ_n\}$, the automorphism group of the cyclic product of groups $C_n\cG$ decomposes as follows:
$$\mathrm{Aut}(C_n\mathcal{G})
%= \mathrm{Inn} \cdot \mathrm{Loc} 
\simeq C_n\mathcal{G} \rtimes \left( \left( \prod\limits_{i \in \bbZ_n} \mathrm{Aut}(G_i) \right) \rtimes \mathrm{Sym} \right).$$
\end{thmA}
In the previous statement, $\mathrm{Sym}$ is an explicit subgroup of the automorphism group of the cycle $C_n$.
We refer to Section \ref{section:aut} for the precise statement. In particular, we immediately get a description of the outer automorphism group:

\begin{corB}
For every $n \geq 5$ and every collection of non trivial groups $\mathcal{G}= \{G_i, i\in \bbZ_n\}$,
$$\mathrm{Out}(C_n\mathcal{G}) \simeq \left( \prod\limits_{i \in \bbZ_n} \mathrm{Aut}(G_i) \right) \rtimes \mathrm{Sym} .$$
\end{corB}

In a nutshell, the previous results state that the automorphisms of these graph products are the `obvious' ones.  In the case of right-angled Coxeter groups, such a phenomenon is linked to the so-called \textit{strong rigidity} of these Coxeter groups, and has strong connections with the famous isomorphism problem for Coxeter groups, see for instance \cite{RACGrigidity}. Note that we can take the local groups in our cyclic  product to be other right-angled Coxeter groups. The resulting cyclic product is again a right-angled Coxeter group, and the previous result can thus be interpreted as a form of strong rigidity of these right-angled Coxeter groups \textit{relative to certain of their parabolic subgroups}.

The explicit computation of the automorphism groups of cyclic products can be used to study the subgroups of such automorphism groups. In particular, since satisfying the Tits alternative is a property stable under graph products \cite{AM} and under extensions, we deduce from Theorem A a combination theorem for the Tits Alternative for such automorphism groups: 

\begin{corC}
Let $C_n\cG$ be a cyclic product of $n$ groups with $n \geq 5$. Then the automorphism group $\mathrm{Aut}(C_n\cG)$ satisfies the Tits Alternative if and only if for every $i \in \bbZ_n$ both $G_i$ and $\mathrm{Aut}(G_i)$ satisfy the Tits alternative.
\end{corC}

Another interesting problem is to understand on which space a given group can act non-trivially, a question that relates to the geometric and analytic structure of the group. Two types of actions that are currently the topic of intense research are actions on hyperbolic spaces and actions on Hilbert spaces. In this article, we focus on two related properties: acylindrical hyperbolicity and Property (T) respectively. 

A group is said \textit{acylindrically hyperbolic} if it admits a non-elementary acylindrical action on a hyperbolic space. Acylindrical hyperbolicity was introduced by Osin in \cite{OsinAcyl}, unifying several known classes of groups with `negatively-curved' features. One of the most impressive consequences of the acylindrical hyperbolicity of a group is its \emph{SQ-universality} \cite{DGO} (ie., every countable group embeds into a quotient of the group we are looking at); as a corollary, such a group contains non abelian free subgroups and uncountably many normal subgroups (loosely speaking, it is far from being simple). We refer to \cite{OsinSurvey} for more information about acylindrically hyperbolic groups. 

Very little is known about the acylindrical hyperbolicity of the automorphism group of a graph product, even in the case of right-angled Artin groups. At one end of the RAAG spectrum, $\mathrm{Aut}(\bbZ^n)= \mathrm{GL}_n(\bbZ)$ is a higher rank lattice for $n \geq 3$, and thus does not have non-elementary actions on hyperbolic spaces by a recent result of Haettel \cite{HaettelRigidity}. The situation is less clear for $\mathrm{Aut}(F_n)$: while it is known that $\mathrm{Out}(F_n)$ is acylindrically hyperbolic for $n \geq 2$ \cite{BestvinaFeighnOutFnHyp}, the case of  $\mathrm{Aut}(F_n)$ seems to be open. For right-angled Artin groups whose outer automorphism group is finite, such as right-angled Artin group over atomic graphs, the problem boils down to the question of the acylindrical hyperbolicity of the underlying group, for which a complete answer is known \cite{MinasyanOsinTrees}. 

For general graph products, we obtain the following: 

 \begin{thmD}
Let $C_n\cG$ be a cyclic product of $n$ non trivial finitely generated groups, with $n \geq5$. Then $\mathrm{Aut}(C_n\cG)$ is acylindrically hyperbolic.
\end{thmD}

We actually prove a slightly more general result. We say that a group $G$  \textit{ has its automorphisms determined by a finite set} if there exists a finite subset $S \subset G$ such that the only automorphism of $G$ fixing $S$ pointwise is the identity. For instance, the automorphisms of a finitely generated group are determined by a given finite generating set. The  group $(\bbQ, +)$ is an example of a non-finitely generated group whose automorphisms are determined by a finite set. The previous theorem still holds if we simply assume that each $G_i$ has its automorphisms determined by a finite set. However, this assumption on the automorphisms of the local groups cannot be removed, as we illustrate in Remark \ref{remark:nonacyl}. It is also worth noticing that, on the other hand, $\mathrm{Out}(C_n \mathcal{G})$ is generally not acylindrically hyperbolic, because it contains a finite-index subgroup splitting as a direct product according to Corollary B. Precisely, $\mathrm{Out}(C_n \mathcal{G})$ is acylindrically hyperbolic  only if $\mathrm{Aut}(G_i)$ is finite for all but one index, say $j$, with $\mathrm{Aut}(G_j)$ acylindrically hyperbolic. Indeed, as a consequence of Corollary B and \cite[Lemma 1]{MinasyanOsinErratum}, $\mathrm{Out}(C_n \mathcal{G})$ is acylindrically hyperbolic only if  $\prod\limits_{i \in \bbZ_n} \mathrm{Aut}(G_i)$ is, and an acylindrically hyperbolic group cannot split as a direct product of two infinite groups \cite[Corollary 7.3]{OsinAcyl}.
%, and conversely the direct product of an acylindrically hyperbolic group with a finite group is necessarily acylindrically hyperbolic as well \cite[Lemma 1]{MinasyanOsinErratum}.

%It is worth noticing that, since we get a non trivial action of $\mathrm{Aut}(C_n \mathcal{G})$ on some CAT(0) cube complex, necessarily Kazhdan's property (T) fails (according to ???).

A group has \textit{Kazhdan's property (T)} if every unitary action on a Hilbert space with almost invariant vectors has a non-trivial invariant verctor. Property (T) for a group imposes for instance strong restrictions on the possible homomorphisms starting from that group (for a geometric realisation of this idea, see for example \cite{PaulinOuter}, whose main construction has been very inspiring in other contexts), and plays a fundamental role in several rigidity statements, including the famous Margulis' superrigidity. We refer to \cite{PropT}, and in particular to its introduction, for more information about Property (T).

We only possess a fragmented picture of the status of property (T) for automorphism groups of right-angled Artin groups. At one end of the RAAG spectrum, $\mathrm{Aut}(\bbZ^n)= \mathrm{GL}_n(\bbZ)$ is known to have property $(T)$ for $n \geq 3$. The situation is less clear for $\mathrm{Aut}(F_n)$: It is known that this automorphism group does not have property (T) for $n =2$ and $3$ \cite{McCoolPropTaut, GrunewaldLubotzkyAutF, BogopolskiVikentievAutF}, and has property (T) for $n=5$ by a recent result of \cite{Aut5PropT}, but the general case remains unknown. For right-angled Artin groups whose outer automorphism group is finite, their automorphism groups are known not to have property (T) as the underlying right-angled Artin group is CAT(0) cubical. For other RAAGs in between, certain of their automorphism groups are also known not to have property (T) by a result of \cite{AramayonaMartinezPerezAutRaag}.  To our knowledge, very little is known for more general graph products. We prove the following: 

\begin{thmE}
For every $n \geq5$ and every collection of non trivial groups $\mathcal{G}$, $\mathrm{Aut}(C_n\cG)$ does not have Property (T).
\end{thmE}

We emphasise that this result does not assume any knowledge of the groups constituting the graph product, or the size of its outer automorphism group.\\
%Our global approach is the following. Thinking of a given cyclic product $C_n \mathcal{G}$ as a nonpositively-curved polygon of groups, one finds a natural action of $C_n \mathcal{G}$ on some hyperbolic polygonal complex $X$. It turns out that this complex can be described in a purely algebraic way, so that the action $C_n \mathcal{G} \curvearrowright X$ extends to an action $\mathrm{Aut}(C_n \mathcal{G}) \curvearrowright X$ (when $C_n \mathcal{G}$ is identified with its inner automorphism group). Next, the decomposition provided by Theorem ??? follows almost immediately from the existence of a fundamental domain reduced to a single polygon; and Theorem ??? is deduced by applying a general criterion proved in ??? to the action of $\mathrm{Aut}(C_n \mathcal{G})$ on the \emph{Davis complex} $X'$, which is a CAT(0) square complex obtained by subdividing the polygonal complex $X$.

We now present our approach. Graph products of groups naturally act on their  \emph{Davis complex}, a CAT(0) cube complex whose faces correspond to cosets of parabolic subgroups. Our main goal is to show that the action of the graph product on its Davis complex extends to an action of the automorphism group. It is not clear  that an automorphism of the group should induce an automorphism of the Davis complex. Indeed, the definition of parabolic subgroups of a graph product does not make them \textit{a priori}  invariant under automorphisms. Worse, in the case of right-angled Artin group over general graphs, the existence of transvections and partial conjugations even shows that parabolic subgroups are not preserved by automorphisms in general. However, focusing on graph products over long cycles prevents such a behaviour. In order to extend the action in our situation, we provide a new algebraic description of the Davis complex that is invariant under automorphisms. In particular, we describe the various parabolic subgroups in terms of families of subgroups invariant under automorphisms: For instance, some of the subgroups involved are stabilisers of certain subcomplexes of the Davis complex known as \textit{tree-walls}, first introduced by Bourdon in the context of quasi-isometric rigidity \cite{Bourdon1997}. While we provide a new algebraic description of the Davis complex, our approach and the techniques used are almost entirely geometric, and rely heavily on both the CAT(0) cubical geometry and the small cancellation geometry of the Davis complex.

Once we know that the action of the graph product extends to an action of its automorphism group, several of our main results follow: Theorem E follows from the existence of a non-trivial action on a CAT(0) cube complex by \cite{NibloRollerPropT}, while Theorem A follows from a further geometric study of the action. In order to prove Theorem D, we use a recent criterion due to Chatterji--Martin \cite{Chatterji_Martin_Criterion_Acyl_Hyp} that allows one to prove the acylindrical hyperbolicity of a group acting on a CAT(0) cube complex.\\

%Such a statement should be compared to ???, ???, ???. 

The paper is organised as follows. We begin, in Section 1, by giving a few general definitions and statements about graph products, before recalling the definition of the complex at the centre of our approach: the Davis complex of a graph product of groups. We also recall results about small cancellation geometry, which will be used to study the geometry of the Davis complex. Section 2 is dedicated to the study of the main tool of our article, namely tree-walls in Davis complexes. The properties of these tree-walls are  exploited  in Section 3 to prove a purely algebraic characterisation of the Davis complex. Finally, Theorem A is proved in Section 4 and Theorem D in Section~5.

\paragraph{Acknowledgements.} The first author thanks the university of Vienna for its hospitality in December 2015, where this project originated. The second author thanks the MSRI for its support during the programme `Geometric Group Theory' in 2016, Grant  DMS-1440140, and the Isaac Newton Institute for its support during the programme `Non-positive curvature group actions and cohomology' in 2017, EPSRC Grant  EP/K032208/1, during which part of this research was conducted. The second author was partially supported by the ERC grant 259527 of G. Arzhantseva  and by the Lise Meitner FWF project M 1810-N25 of the second author.

%% file: Preliminaries.tex
\section{Preliminaries}

In this section, we recall standard facts about graph products of groups, certain polyhedral complexes  on which they act, and the geometry of such complexes.

\subsection{Graph products}\label{section:graphproducts}

\noindent
Given a simplicial graph $\Gamma$, whose set of vertices is denoted by $V(\Gamma)$, and a collection of groups $\mathcal{G}=\{ G_v \mid v \in V(\Gamma) \}$, we define the \emph{graph product} $\Gamma \mathcal{G}$ as the quotient
\begin{center}
$\left( \underset{v \in V(\Gamma)}{\ast} G_v \right) / \langle \langle gh=hg, \ h \in G_u, g \in G_v, \{ u,v \} \in E(\Gamma) \rangle \rangle$,
\end{center}
where $E(\Gamma)$ denotes the set of edges of $\Gamma$. Loosely speaking, it is obtained from the disjoint union of all the $G_v$'s, named the \emph{vertex-groups}, by requiring that two adjacent vertex-groups  commute. Notice that, if $\Gamma$ has no edges, $\Gamma \mathcal{G}$ is the free product of the groups of $\mathcal{G}$; on the other hand, if $\Gamma$ is a complete graph, then $\Gamma \mathcal{G}$ is the direct sum of the groups of $\mathcal{G}$. Therefore, a graph product may be thought of as an interpolation between free and direct products. Graph products also include two classical families of groups: If all the vertex-groups are infinite cyclic, $\Gamma \mathcal{G}$ is known as a \emph{right-angled Artin group}; If all the vertex-groups are cyclic of order two, then $\Gamma \mathcal{G}$ is known as a \emph{right-angled Coxeter group}.

%\heyAlex{Je trouve la notation $\Gamma\cG$ d' Osin absolument horrible - et une \'ecriture fonctorielle comme \c{c}\`a me semble n'avoir aucun sens. Vu que leur notation est loin d'\^etre universelle, tu ne pr\'ef\`ererais pas une notation du style $G_\Gamma$ (comme le fait Charney), ce qui rappelle d'ailleurs la notation $A_\Gamma$ pour les RAAGs ? Auquel cas, il faudrait juste changer certaines conventions, du genre : $\cG$ est un ensemble de groupes, et $\Gamma$ un graphe \'etiquett\'e dont les sommets sont \'etiquett\'es par des \'el\'ements de $\cG$. 
%	
%	Ind\'ependamment de cela, il serait sans doute bien de fixer une fois pour toutes $\cG$, $n$ et $\Gamma$ dans cette section, et de poser  $G:= \Gamma\cG$ ou $G_\Gamma$, les notations dans le reste de l'article sont un peu lourdes, sinon.}
	
%	\heyAnthony{C'est vrai qu'il n'y a pas de notation universelle, mais $\Gamma \mathcal{G}$ est celle que je pr\'ef\`ere, parce qu'elle est plus concise que les quelques autres qui ont \'et\'e introduites. Contrairement \`a $G_{\Gamma}$, elle permet \'egalement d'\'eviter les indices doubles, $C_n \mathcal{G}$ au lieu de $G_{C_n}$. Un compromis pourrait \^etre $G(\Gamma)$ ? Cela dit, ma pr\'ef\'erence reste \`a $\Gamma \mathcal{G}$, parce que le $G$ de $G(\Gamma)$ pourrait se m\'elanger avec d'aures $G$. Et puis cela multiplie les parenth\`eses, comme dans $\mathrm{Aut}(G(\Gamma))$.
%	
%	Pour ta deuxi\`eme remarque, je trouve que l'avantage d'avoir une notation concise permet justement d'\'eviter de devoir poser $G$ pour le groupe.}

\medskip \noindent
\textbf{Convention.} In all the article, we will assume for convenience that the groups of $\mathcal{G}$ are non-trivial. Notice that it is not a restrictive assumption, since a graph product with some trivial factors can be described as a graph product over a smaller graph all of whose factors are non-trivial.

\medskip \noindent
A \emph{word} in $\Gamma \mathcal{G}$ is a product $g_1 \cdots g_n$ where $n \geq 0$ and where, for every $1 \leq i \leq n$, $g_i$ belongs to some subgroup of $\mathcal{G}$. The $g_i$'s are the \emph{syllables} of the word, and $n$ is the \emph{length} of the word. Clearly, the following operations on a word do not modify the element of $\Gamma \mathcal{G}$ it represents:
\begin{itemize}
	\item[(O1)] delete the syllable $g_i=1$;
	\item[(O2)] if $g_i,g_{i+1} \in G$ for some $G \in \mathcal{G}$, replace the two syllables $g_i$ and $g_{i+1}$ by the single syllable $g_ig_{i+1} \in G$;
	\item[(O3)] if $g_i$ and $g_{i+1}$ belong to two adjacent vertex-groups, switch them.
\end{itemize}
A word is \emph{reduced} if its length cannot be shortened by applying these elementary moves. 
%In practice, if $g=g_1 \cdots g_n$ is a reduced word and $h$ is a syllable, then a reduction of the product $gh$ is given by
%\begin{center}
%$\left\{ \begin{array}{ll} g_1 \cdots g_n & \text{if}~h=1 \\ g_1 \cdots g_{i-1} \cdot g_{i+1} \cdots g_n & \text{if}~g_i~ \text{shuffles to the end and}~g_i=h^{-1} \\ g_1 \cdots g_{i-1} \cdot g_{i+1} \cdots g_n \cdot (g_ih) & \text{if}~g_i~ \text{shuffles to the end and}~g_i \neq h^{-1} \end{array} \right.$
%\end{center}
%In particular, 
Every element of $\Gamma \mathcal{G}$ can be represented by a reduced word, and this word is unique up to applying the operation (O3); see \cite{GraphProduct} for more details. 
%This uniqueness allows us to define the \emph{support} of an element $g \in \Gamma \mathcal{G}$, denoted by $\mathrm{supp}(g)$, as the set of the vertices of $\Gamma$ corresponding to the vertex-groups which contain the syllables of some reduced word representing $g$. 
We mention the following observation for future use: 

\begin{cor}\label{cor:uniquereduced}
 Let $g:= g_1 \cdots g_k$ be a word where  no two consecutive syllables $g_i, g_{i+1}$ belongs to the same group of $\cG$ nor to groups of $\cG$ that are joined by an edge of $\Gamma$. Then the word $g= g_1 \cdots g_k$ is the \textit{unique} reduced form of $g$.\qed
\end{cor}

If $\Lambda$ is an \emph{induced} subgraph of $\Gamma$ (ie., two vertices of $\Lambda$ are adjacent in $\Lambda$ if and only if they are adjacent in $\Gamma$), then the subgroup, which we denote by $\Lambda \mathcal{G}$, generated by the vertex-groups corresponding to the vertices of $\Lambda$ is naturally isomorphic to the graph product $\Lambda \mathcal{G}_{|\Lambda}$, where $\mathcal{G}_{|\Lambda}$ denotes the subcollection of $\mathcal{G}$ associated to the set of vertices of $\Lambda$. A \emph{join subgroup} of $\Gamma \mathcal{G}$ is a subgroup conjugated to $\Lambda \mathcal{G}$ for some \emph{join} $\Lambda \subset \Gamma$ (ie., $\Lambda$ contains two induced subgraphs $\Lambda_1$ and $\Lambda_2$, covering all the vertices of $\Lambda$, so that any vertex of $\Lambda_1$ is adjacent to any vertex of $\Lambda_2$).

\medskip \noindent
The following lemma is very useful in understanding subgroups of $\Gamma \mathcal{G}$.

\begin{lem}[{\cite[Corollary 6.15]{MinasyanOsinTrees}}]\label{lem:centralizer}
Let $H$ be a subgroup of $\Gamma \mathcal{G}$. Then either	$H$ is contained in a join subgroup or it contains an element whose centraliser is infinite cyclic. \qed
\end{lem}

Here are some immediate consequences.

\begin{cor}\label{join}
Any subgroup of $\Gamma \mathcal{G}$ which splits non-trivially as a direct product is contained in a join subgroup. \qed
\end{cor}

\begin{cor}\label{center}
	If $\Gamma$ is not a join and contains at least two vertices, then the centre of  $\Gamma \mathcal{G}$ is trivial. \qed
\end{cor}

%\begin{proof}
%Let $H$ be a subgroup of $\Gamma \mathcal{G}$ which splits non-trivially as a direct product. We want to prove that $H$ is included into a join subgroup. If $H$ is included into the conjugate of a vertex-group, the conclusion follows. Otherwise, \cite[Corollary 6.15]{MinasyanOsinTrees} implies that either $H$ is included into a join subgroup or it contains an element whose centraliser is infinite cyclic. Because $H$ splits non-trivially as a direct product, the latter case cannot happen, so the conclusion follows.
%\end{proof}

\noindent
We also mention the following result, which will be used in Section $3$.

\begin{lem}[{\cite[Proposition 3.13]{AM}}]\label{normalizer}
Let $\Lambda$ be an induced subgraph of $\Gamma$, and denote by $\Xi$ the induced subgraph of $\Gamma$ generated by $\Lambda$ and the vertices of $\Gamma$ adjacent to all the vertices of $\Lambda$. Then the normalizer of $\Lambda \mathcal{G}$ in $\Gamma \mathcal{G}$ is $\Xi \mathcal{G}$.\qed
\end{lem}

%\noindent
%Let $g \in \Gamma \mathcal{G}$. We define the \emph{tail} of $g$, denoted by $\mathrm{tail}(g)$, as the set of the vertices of $\Gamma$ corresponding to the vertex-groups which contain the last syllables of the reduced words representing $g$. Similarly, we define the \emph{head} of $g$, denoted by $\mathrm{head}(g)$, as the set of the vertices of $\Gamma$ corresponding to the vertex-groups which contain the first syllables of the reduced words representing $g$. Alternatively, notice that $\mathrm{head}(g)= \mathrm{tail}(g^{-1})$.

\noindent
In this article, we will be interested in the case where $\Gamma$ is a cycle $C_n$ of length $n \geq 5$. We call such a graph product a \textit{cyclic product} of groups. For convenience, the vertex-groups will be denoted by $G_1, \ldots, G_{n}$ such that $G_i$ is adjacent to $G_{i-1}$ and $G_{i+1}$, where indices are to be understood modulo $n$. We recall that we assume, by convention, that the $G_i$'s are non trivial.

\subsection{The Davis complex of a cyclic product of groups}
\label{Davis} 

In this section, we recall a construction due to Davis -- valid for every graph product of groups -- and study it in the context of cyclic products. We start from an arbitrary graph product of groups $\Gamma\cG$.

\begin{definition}[Davis complex]
	We define the \textit{Davis complex} of a graph product as follows: 
	\begin{itemize}
		\item Vertices correspond to left cosets of the form $g\Lambda \cG$ for $g\in \Gamma\cG$ and $\Lambda \subset \Gamma$ a complete subgraph.
		\item For every $g\in \Gamma\cG$ and complete subgraphs $\Lambda_1, \Lambda_2 \subset \Gamma$ that differ by exactly one vertex, one puts an edge between the vertices $g\Lambda_1\cG$ and $g\Lambda_2\cG$.
		\item One obtains a cube complex from this graph by adding for every $k \geq 2$ a $k$-cube for every  subgraph isomorphic to the $1$-skeleton of a $k$-cube.
		\end{itemize}
This complex comes with an action of $\Gamma\cG$: The group $\Gamma\cG$ acts on the vertices by left multiplication on  left cosets, and this action extends to the whole complex.
	\end{definition}

\begin{thm}[{\cite[Example II.12.30(2)]{BH}}]
The Davis complex of a graph product is a CAT(0) cube complex.
%The previous complex, called the \textit{Davis complex} of the cyclic product, is a CAT(0) cube complex.
\qed
\end{thm}

From now on, we fix an integer $n \geq 5$ and a collection $\cG= \{G_i, i \in \bbZ_n\}$ of (non trivial) groups, and we consider the cyclic product $C_n\cG$. In this case, the maximal complete subgraphs of $C_n$ are edges, hence the resulting Davis complex is a CAT(0) square complex. More precisely, there are three types of vertices: cosets of the trivial subgroup, cosets of the form $gG_i$, and cosets of the form $g(G_i \times G_{i+1})$ ($g \in C_n\cG$ and $i \in \bbZ_n$).

Note that the Davis complex of this cyclic product is naturally the cubical subdivision of a polygonal complex made of $n$-gons, obtained as follows: vertices are the left cosets of the form $g(G_i \times G_{i+1})$; for every $g \in C_n\cG$ and $i \in \bbZ_n$ we add an edge between  $g(G_i \times G_{i+1})$ and $g(G_{i+1} \times G_{i+2})$; and for every $g \in C_n\cG$ we add a polygon with boundary the $n$-cycle $g(G_1 \times G_{2}), g(G_2 \times G_{3}), \ldots, g(G_n \times G_{1})$. 

\medskip \noindent
\textbf{Notation.} We will denote by $X$ this underlying polygonal complex, and by $X'$ the Davis complex, which we identify with the first cubical subdivision of $X$.\\

%\heyAlex{Maybe add a picture of a hexagon and the associated cubical subdivision for clarity.}

We  mention here a few useful observations about the action of $C_n\cG$ on $X$.

\begin{obs}\label{obs:stab} Stabilisers for the action behave as follows:
	\begin{itemize}
		\item The stabiliser of a vertex of $X$ is conjugated to a subgroup of the form $G_i \times G_{i+1}$.
		 \item The stabiliser of an edge of $X$ is conjugated to a subgroup of the form $G_i$.
		 \item The stabiliser of a polygon of $X$ is trivial. \qed
		\end{itemize}
	\end{obs}

We mention here a few elementary results about the geometry of $X$.
\begin{obs} We have the following:
	\begin{itemize}
	\item The polygonal complex $X$ is a cocompact $C(n)$-$T(4)$ complex. In particular it is Gromov-hyperbolic since $n \geq 5$.
			\item Two polygons of $X$ intersect in at most one edge. \qed
\end{itemize}
	\end{obs}

\subsection{The combinatorial Gau\ss-Bonnet theorem}

%
%We will see in the next section that the intersection complex $\Delta_X$ is a \textit{small cancellation} complex. We now recall definitions and known results about such complexes. 
%	
%\begin{definition}
%	Let $Y$ be a polygonal complex. A \textit{piece} of $Y$ is a path $\gamma$ such that there exist two polygons $P_1$ and $P_2$ of $Y$ such that the map $\gamma \ra Y$ factors as $\gamma \ra P_1 \ra Y$ and $\gamma \ra P_2 \ra Y$ but there does not exist a homeomorphism $\partial P_1 \ra \partial P_2$ such that  the following diagram commutes: 
%	$$ \xymatrix{
%		\gamma  \ar[d]_{} \ar[r]_{}^{} & \partial P_2 \ar[d]^{} \\
%		\partial P_1 \ar[r]_{} \ar[ur]& Y. \\
%	}$$\\
%	By convention, edges of $Y$ are also considered as pieces.
%	
%	We say that $Y$ is a $C'(1/6)$ \textit{polygonal complex} if, for every piece $\gamma$ of $Y$ and every polygon $P$ of $Y$ containing $\gamma$, we have:
%	$$|\gamma| < \frac{1}{6} \cdot |\partial P|.$$
%	
%	We say that $Y$ is a $T(4)$ \textit{polygonal complex} if links of vertices have girth at least $4$.
%	\label{smallcancellationcomplex}
%\end{definition}
%
%\begin{obs}
%	If $n \geq 7$, the  polygonal complex $X$ defined in Section \ref{polygon_of_groups} is a $C'(1/6)-T(4)$ polygonal complex. \qed
%\end{obs}

We now recall the main tool used to study the geometry of such complexes.\\

%\noindent \textbf{Disc diagrams.}  
A \textit{disc diagram} $D$ over a polygonal complex $Y$ is a finite contractible planar CW-complex, together with a cellular map  $D \ra Y$ which restricts to a homeomorphism on every closed $2$-cell. We will say that a disc diagram is \textit{reduced} if no two distinct $2$-cells of $D$ that share a $1$-cell are mapped to the same polygon of $Y$. A disc diagram is said to be \textit{non-degenerate} if its boundary is homeomorphic to a circle, and \textit{degenerate} otherwise. A vertex of $D$ is said to be \textit{internal} if its link is homeomorphic to a circle, and is a \textit{boundary vertex} otherwise.

Recall that, by the Lyndon--van Kampen Theorem, one can associate to every null-homotopic and non-backtracking loop $\gamma: S \ra Y$ a reduced disc diagram $D \ra Y$ whose restriction to the boundary is the given loop. We will say that the disc diagram $D \ra Y$ admits $\gamma$ as boundary, or that $D$ fills $\gamma$. 

%\begin{convention}\label{conv:disc}
%Let $D$ be a contractible planar  polygonal complex. Up to ``removing vertices'' of $D$, we can always assume that $D$ does not contain:
%\begin{itemize}
%\item internal vertices which are contained in exactly two $2$-cells of $D$, or
%\item boundary vertices which do not disconnect $D$ and are contained in exactly one $2$-cell of $D$.  
%\end{itemize}
%In everything that follows, we will always assume that $D$ does not contain any such vertex. With such a convention, we will refer to a $1$-cell contained in the boundary of some $2$-cell $P$ of $D$ as a \textit{side} of $P$. For a disc diagram $D \ra X$, we emphasise that such a side is mapped to a \textit{concatenation} of edges of $X$.
%\end{convention}

%\noindent \textbf{Curvatures.}
 Let $D$ be a planar contractible  polygonal complex. The \textit{curvature} of a vertex $v$ of $D$ is defined as:
$$\kappa_D(v) = 2\pi - \pi \cdot \chi(\mbox{link}(v)) - n_v \frac{\pi}{2},$$
where $n_v$ denotes the number of closed $2$-cells of $D$ containing $v$.

The  \textit{curvature} of a closed $2$-cell $f$ of $D$ is defined as
$$\kappa_D(f) = 2\pi-n_f\frac{\pi}{2},$$
where $n_f$ denotes the number of vertices of $D$ contained in $f$.

The following version of the combinatorial Gau\ss--Bonnet Theorem, which follows the presentation of McCammond--Wise \cite[Theorem 4.6]{McCammondWiseFansLadders}, is a powerful tool in controlling the geometry of disc diagrams.

\begin{thm}[Combinatorial Gau\ss-Bonnet Theorem {\cite[Theorem 4.6]{McCammondWiseFansLadders}}]\label{thm:Gauss_Bonnet} Let $D$ be a planar contractible polygonal complex. Then:
\begin{equation*}
~~~~~~~~~~~~~~~~~~~~~~~~~~~~~~~~~~~~\underset{v}{\sum} \kappa_D(v) + \underset{f}{\sum} \kappa_D(f) = 2\pi. ~~~~~~~~~~~~~~~~~~~~~~~~~~~~~~~~~~~~\qed
\end{equation*}
\label{GaussBonnet}
\end{thm}

%% file: Walls.tex
\section{Tree-walls of $X$}

\subsection{Geometric  properties}

\medskip \noindent
In this section, we are interested in specific subgraphs in $X^{(1)}$, called \emph{tree-walls}, which will play a fundamental role in Section \ref{section:algcharac}. These subspaces have been introduced for the first time by Bourdon in \cite{Bourdon1997}.

\begin{definition}
The \textit{label} of an edge of $X$ is the unique integer $i\in \bbZ_n$ such that the stabiliser of that edge is conjugated to $G_i$ (see Observation \ref{obs:stab}).
\end{definition}

\begin{definition}
A \textit{tree-wall} of $X$ is a maximal connected subgraph of $X^{(1)}$ whose edges all have the same label. 
\end{definition}

\noindent
When $X$ is endowed with its CAT(0) metric, tree-walls turn out to be convex subspaces. As a consequence, they are contractible graphs, ie., trees. This statement essentially follows from the following observation:

\begin{lem}
Two adjacent edges of a tree-wall make an angle $\pi$.
\end{lem}

\begin{proof}
For two adjacent edges of a tree-wall, it is enough to show that the corresponding adjacent edges of $X'$ make an angle $\pi$. If they made an angle $\pi/2$, they would be contained in a square of $X$, which, by construction of the Davis complex, would imply that they have different labels.
\end{proof}

%\noindent
%\textcolor{red}{(A quick argument?)}

\medskip \noindent
By combining the previous observation and some elementary results about CAT(0) geometry, the desired conclusion follows:

\begin{cor}\label{cor:subtree}
Tree-walls are convex subtrees of $X$. \qed
\end{cor}

It turns out that the collection of tree-walls of $X$ is in bijection with the set of parallelism classes of hyperplanes in the CAT(0) square complex $X'$. The desired bijection is given by the following lemma. We recall that a \emph{combinatorial hyperplane} of a given hyperplane $\hat{h}$ is one of the two connected components of $N( \hat{h} ) \backslash \backslash \hat{h}$, ie., the subcomplex obtained from the carrier $N ( \hat{h} )$ of $\hat{h}$ by removing the interiors of all the edges dual to $\hat{h}$.
%First notice that a hyperplane $\hat{h}$ of $X'$ is associated to an edge of $X'$ corresponding to an inclusion of the form $e \subset P$, where $e$ is an edge of $X$ and $P$ a polygon of $X$.

\begin{lem}\label{lem:wall_hyperplane}
	Let $\hat{h}$ be a hyperplane of $X'$. Then exactly one of its combinatorial hyperplanes is a tree-wall. 
\end{lem}

%\begin{lem}\label{lem:wall_hyperplane}
%Let $\hat{h}$, $e$, $P$ be as above. Then the hyperplane $\hat{h}$ is at bounded Hausdorff distance from the tree-wall $T_e$ associted to $e$.
%\end{lem}

\noindent

\begin{proof}[Proof of Lemma \ref{lem:wall_hyperplane}.]
This is a direct consequence of the fact that links of vertices of $X$ are complete bipartite graphs with respect to the labelling.
\end{proof}

\begin{definition}
Let $\hat{h}$ be a hyperplane of $X'$. The tree-wall that is a combinatorial hyperplane of $\hat{h}$ will be called the  tree-wall \textit{associated to} $\hat{h}$.
\end{definition}

\subsection{Stabilisers of tree-walls}

We gather here results about pointwise and global stabilisers of tree-walls and  their intersections.

\begin{lem}\label{lem:wallfix}
	Let $T$ be a tree-wall of $X$ and let $e \subset T$ an edge. Then $\mathrm{fix}(T)= \mathrm{stab}(e)$. 
\end{lem}

\begin{proof}
	The inclusion $\mathrm{fix}(T)\subset \mathrm{stab}(e)$ being clear, we prove the reverse inclusion. Let $e_1, e_2$ be two adjacent edges of $T$, and let $v$ be their common vertex. There exists an element $g \in \mathrm{stab}(v)$ such that $ge_1 = e_2$. But since $\mathrm{stab}(e_1)$ is a direct factor (hence a normal subgroup) of $\mathrm{stab}(v)$, we have $\mathrm{stab}(e_1) = \mathrm{stab}(e_2)$. A proof by induction now implies that $\mathrm{stab}(e) = \mathrm{stab}(e')$ for every other edge $e'$ of $T$, and it follows that  $\mathrm{fix}(T)\supset \mathrm{stab}(e)$. 
	%$u \in C_n$ be the vertex labelling tha wall $A$. Up to translating $e$, we may suppose without loss of generality that $\mathrm{stab}(e)=G_u$. Because there exists a unique $C_n\mathcal{G}$-orbit of edges labelled by $u$, and since two walls either coincide or are disjoint, we know that $A$ has a unique $\mathrm{stab}(A)$-orbit of edges. But $\mathrm{stab}(A)= \langle G_u,G_v,G_w \rangle$, if $v,w \in C_n$ denote the two vertices adjacent to $v$. By noticing that $G_u$ commutes with both $G_v$ and $G_w$, it follows that all the edges of $A$ have $G_u$ as their stabilisers, so that $G_u= \mathrm{stab}(e)$ fixes $A$. 
\end{proof}

\begin{lem}\label{lem:stab_treewall}
Let $T$ be a tree-wall of $X$ with label $i \in \bbZ_n$. Then $\mathrm{stab}(T)$ is conjugated to $\langle G_{i-1}, G_i, G_{i+1}\rangle$.
\end{lem}

\begin{proof}
First notice that every element of the subgroups $G_{i-1}$, $G_i$ and $G_{i+1}$ stabilises $T$, hence  $\langle G_{i-1}, G_i, G_{i+1}\rangle \subset \mathrm{stab}(T)$. We now prove the reverse inclusion. 
	
Up to the action of the group, let us assume that $T$ coincides with the tree-wall $T_e$ which contains the edge $e$ of $X$ corresponding to the coset $G_i$. Let $u, v$ be the two vertices of $e$, such that $\mathrm{stab}(u) = \langle G_{i-1}, G_i \rangle $ and $\mathrm{stab}(v) = \langle G_{i}, G_{i+1}\rangle $. Note that $\mathrm{stab}(u)$ acts transitively on the edges of $T_e$ containing $u$, and $\mathrm{stab}(v)$ acts transitively on the edges of $T_e$ containing $v$. An induction now implies that $\langle \mathrm{stab}(u), \mathrm{stab}(v)\rangle $ acts transitively on the edges of $T_e$.

Let $g\in \mathrm{stab}(T_e)$. There exists an element of $h\in \langle \mathrm{stab}(u), \mathrm{stab}(v)\rangle  =  \langle G_{i-1}, G_i, G_{i+1}\rangle$ such that $ge = he$. We thus have $gh^{-1} \in \mathrm{stab}(e) = G_i$, hence $g \in \langle G_{i-1}, G_i, G_{i+1}\rangle$.
\end{proof}

We will now focus on stabilisers of pairs of tree-walls. In order to formulate our result, we first introduce the following construction:

\begin{definition}[Crossing graph]
The \textit{crossing graph} of $X$ is the simplicial graph defined as follows: 
\begin{itemize}
\item vertices correspond to tree-walls of $X$,
\item two vertices are joined by an edge if and only if the corresponding tree-walls intersect.
\end{itemize}
We will denote by $\Delta$ the induced graph metric on the intersection graph.
\end{definition}

\begin{prop}\label{prop:interstabwall}
Let $T, T' \subset X$ be two tree-walls. Then $\mathrm{stab}(T) \cap \mathrm{stab}(T')$ is:
\begin{itemize}
	\item trivial if $\Delta(T, T') \geq 3$;
	\item  conjugated to some $G_i$, for some $i \in \bbZ_n$, if $\Delta(T,T')=2$,
	\item  conjugated to $\langle G_i,G_{i+1} \rangle$, for some $i \in \bbZ_n$, if $\Delta(T,T')=1$.
\end{itemize}
\end{prop}

%\begin{lem}
%There exist at most one wall intersecting two disjoint given walls. 
%\end{lem}

%\begin{proof}
%...
%\end{proof}

We will split the proof in three separate cases. We start with tree-walls at distance $1$. Notice that since tree-walls are convex subtrees of $X$ by Corollary \ref{cor:subtree}, we immediately have the following:

\begin{lem}
Two  tree-walls at $\Delta$-distance $1$ intersect in a single vertex. \qed
\end{lem}

\begin{cor}\label{cor:stab1}
Let $T, T'$ be two tree-walls of $X$ at $\Delta$-distance $1$. Then $\mathrm{stab}(T) \cap \mathrm{stab}(T')$  is  conjugated to some $\langle G_i,G_{i+1} \rangle$, for some $i \in \bbZ_n$.
\end{cor}

\begin{proof}
Since two edges with different labels are not in the same $C_n\cG$-orbit, it follows that a group element is in  $\mathrm{stab}(T) \cap \mathrm{stab}(T')$ if and only if it stabilises their unique intersection point $T \cap T'$.
\end{proof}

We now move to the case of tree-walls at distance $2$. We start by the following lemma:

\begin{lem}\label{lem:Delta_2}
Let $T, T'$ be two tree-walls of $X$ at $\Delta$-distance $2$. Then there exists a unique combinatorial path of minimal length between $T$ and $T'$, contained in the (unique) tree-wall crossing $T$ and $T'$.
\end{lem}
 
\begin{proof}
Let $T_0$ be a tree-wall crossing both $T$ and $T'$, and let $\gamma_0$ be the portion of $T_0$ between $T$ and $T'$. By contradiction, suppose that there exists a combinatorial geodesics $\gamma$ from $T$ to $T'$ distinct from $\gamma_0$, and such that $|\gamma| \leq |\gamma_0|$. Without loss of generality, we assume that $\gamma$ intersects $T$ in a single vertex. First notice that since $\gamma_0$ is the unique CAT(0) geodesic between $T_0 \cap T$ and $T_0 \cap T'$, $\gamma_0$ and $\gamma$ cannot have the same endpoints. Without loss of generality, suppose that $\gamma \cap T \neq \gamma_0 \cap T$, and let $\tau$ be the unique geodesic of $T$ between  $\gamma \cap T$ and $ \gamma_0 \cap T$. There are two cases to consider here: If $\gamma \cap \gamma_0\neq \varnothing$, one then constructs an embedded loop of $X$ as the concatenation of a subsegment $\gamma_0'$ of $\gamma_0$, $\tau$, and a subsegment $\gamma'$ of $\gamma$. If $\gamma \cap \gamma_0 = \varnothing$, one then constructs an embedded loop of $X$ as the concatenation of $\gamma_0$, $\tau$, $\gamma$, and the unique geodesic $\tau'$ of $T'$ between $\gamma \cap T'$ and $\gamma_0 \cap T'$. In the latter case, we set $\gamma' := \gamma$ and $\gamma_0' := \gamma_0$: This will allow us to derive a contradiction for both cases in a uniform manner. 

Since $X$ is simply connected, we choose a reduced disc diagram filling the aforementioned loop. Note that since two consecutive edges of a tree-wall make an angle $\pi$, the only boundary vertices of the disc diagram contributing to positive curvature (exactly $\pi/2$) are the intersection points $T \cap T_0$, $T' \cap T_0$,  and possibly vertices of $\gamma'$: at most $|\gamma'|+1$ other vertices of $\gamma'$. This leads to a maximal contribution of at most $2\pi + (|\gamma'|-1)\frac{\pi}{2}$. The interior vertices contribute to non-positive curvature, and each polygon contributes to a negative curvature bounded above by $-\pi/2$ (since $n \geq 5$). But since two consecutive edges of $T_0$  make an angle $\pi$ and polygons of $X$ are combinatorially convex, no two edges of $T_0$ belong to the same polygon, hence the disc diagram contains at least $|\gamma_0|$ polygons, and thus there the negative curvature contribution is at least $-|\gamma_0|\frac{\pi}{2}$. Since $|\gamma| \leq |\gamma_0 |$ by assumption, we get that the total sum of curvatures is bounded above by $2\pi - \frac{\pi}{2}$, contradicting the Gau\ss--Bonnet Theorem.
\end{proof}

\begin{cor}\label{cor:stab2}
Let $T, T'$ be two tree-walls of $X$ at $\Delta$-distance $2$. Then $\mathrm{stab}(T) \cap \mathrm{stab}(T')$  is  conjugated to some $G_i$, for some $i \in \bbZ_n$.
\end{cor}

\begin{proof}
Let $T_0$ be the tree-wall crossing $T$ and $T'$ (which is unique according to by Lemma \ref{lem:Delta_2}), and let $\gamma_0$ be the portion of $T_0$ between $T$ and $T'$. As a consequence of Lemma \ref{lem:Delta_2}, elements of $\mathrm{stab}(T) \cap \mathrm{stab}(T')$ pointwise stabilise $\gamma_0$. We actually have the reverse inclusion, namely $\mathrm{stab}(T) \cap \mathrm{stab}(T') = \mathrm{fix}(\gamma_0)$, since $C_n\cG$ acts by label-preserving isomorphisms. By Lemma \ref{lem:wallfix}, we have for any edge $e_0$ of $T_0$ the equality $\mathrm{fix}(\gamma_0) = \mathrm{fix}(T_0) = \mathrm{stab}(e_0)$, and such a stabiliser is conjugated to some $G_i$ for some $i \in \bbZ_n$. 
\end{proof}

We finally move to the case of tree-walls at $\Delta$-distance at least $3$.

\begin{lem}\label{lem:min_bounded}
Let $T,T'$ be two walls of $X$. The the subset $\mathrm{Min}_T(T,T')$ of points of $T$ realising the combinatorial distance between $T$ and $T'$ is bounded.
\end{lem}

\begin{proof}
Denote by $k$ the combinatorial distance between $T$ and $T'$. We prove by contradiction that $\mathrm{Min}_T(T,T')$ has diameter at most $2k$. Suppose that there exist vertices $x, y \in T$ and $x', y' \in T'$ with $d(x, x') = k$, $d(y, y') = k$ and $d(x, y) \geq 2k+1$. Let $\gamma$ be the portion of $T$ between $x$ and $y$, let $\gamma'$ be the portion of $T'$ between $x'$ and $y'$, and let $\gamma_{x, x'}, \gamma_{y, y'}$ be combinatorial geodesics between $x$ and $x'$, $y$ and $y'$ respectively.  Since $\gamma_{x, x'} \cap \gamma_{y, y'} = \varnothing$ by construction, one constructs an embedded loop of $X$ as the concatenation of $\gamma$, $\gamma_{x, x'}$, $\gamma'$, $\gamma_{y, y'}$. 
%In the latter case, we set $\gamma_{x, x'}' := \gamma_{x, x'}$ and $\gamma_{y, y'}' := \gamma_{y, y'}$: This will allow us to derive a contradiction for both cases in a uniform manner.

Since $X$ is $T(4)$ and two consecutive edges of $T$ or $T'$ make an angle $ \pi$, there are at most $2k+2$ vertices contributing positive curvature, namely $\pi/2$. But since polygons are combinatorially convex and two consecutive edges of $T$ or $T'$ make an angle $ \pi$, there are at least $2k$ polygons contributing to a negative curvature bounded above by $-\pi/2$. The total sum of curvature is thus at most $\pi$, contradicting the Gau\ss--Bonnet Theorem.
\end{proof}

\begin{cor}\label{cor:stab3}
Let $T, T'$ be two tree-walls of $X$ at $\Delta$-distance at least $3$. Then $\mathrm{stab}(T) \cap \mathrm{stab}(T')$  is trivial.
\end{cor}

\begin{proof}
Elements in $\mathrm{stab}(T) \cap \mathrm{stab}(T')$ also stabilise $\mathrm{Min}_T(T,T')$ and $\mathrm{Min}_{T'}(T,T')$. By Lemma \ref{lem:min_bounded} these are bounded subsets of a CAT(0) space, it follows that elements of $\mathrm{stab}(T) \cap \mathrm{stab}(T')$ stabilise pointwise the circumcentre of  $\mathrm{Min}_T(T,T')$ and $\mathrm{Min}_{T'}(T,T')$, hence also the unique CAT(0) geodesic $\gamma$ between these circumcentres. If $\gamma$ is not contained in $X^{(1)}$, then $\gamma$ contains points whose stabiliser is trivial, hence 
$\mathrm{stab}(T) \cap \mathrm{stab}(T')$ is trivial. If $\gamma$ is contained in $X^{(1)}$, then since $\Delta(T, T')\geq 3$, then $\gamma$ contains at least two consecutive edges making an angle of $\pi/2$. The stabilisers of two such consecutive edges intersect trivially, hence $\mathrm{stab}(T) \cap \mathrm{stab}(T')$ is trivial. 
\end{proof}

\begin{proof}[Proof of Proposition \ref{prop:interstabwall}.]
This now follows from Corollaries \ref{cor:stab1}, \ref{cor:stab2}, and \ref{cor:stab3}.
\end{proof}

%% file: AlgCharact.tex
\section{Algebraic characterisation of $X$}\label{section:algcharac}

\noindent
In all this section, we fix a cycle $C_n$ of length $n \geq 5$ and a collection of non trivial groups $\mathcal{G}$ indexed by the vertices of $C_n$. As described in Section \ref{Davis}, we denote by $X$ the polygonal Davis complex of $C_n \mathcal{G}$. 

\medskip \noindent
Let $\mathcal{M}$ denote the collection of maximal subgroups of $C_n\mathcal{G}$ which decompose non trivially as direct products, and let $\mathcal{C}$ denote the collection of non trivial subgroups of $C_n \mathcal{G}$ which can be obtained by intersecting two subgroups of $\mathcal{M}$. A subgroup of $C_n \mathcal{G}$ which belongs to $\mathcal{C}$ is
\begin{itemize}
	\item \emph{$\mathcal{C}$-minimal} if it is minimal in $\mathcal{C}$ with respect to the inclusion;
	\item \emph{$\mathcal{C}$-maximal} if it is maximal in $\mathcal{C}$ with respect to the inclusion (ie., if it belongs to $\mathcal{M}$);
	\item \emph{$\mathcal{C}$-medium} otherwise.
\end{itemize}
\begin{definition}Let $\mathscr{X}$ be polygonal complex constructed in the following way:
\begin{itemize}
	\item the vertices of $\mathscr{X}$ are the $\mathcal{C}$-medium subgroups of $C_n \mathcal{G}$;
	\item the edges of $\mathscr{X}$ link two subgroups $H_1$ and $H_2$ if $\langle H_1,H_2 \rangle$ is $\mathcal{C}$-maximal;
	\item the polygons of $\mathscr{X}$ fill in the induced cycles of length exactly $n$.
\end{itemize}
The action of $C_n\cG$  on $\mathscr{X}^{(0)}$ by conjugation extends to an action on $\mathscr{X}$.
\end{definition}

The main statement of this section is that this new polygonal complex turns out to be equivariantly isomorphic to $X$.

\begin{prop}\label{prop:algch}
The map 
$$\left\{ \begin{array}{ccc} X^{(0)} & \to & \mathscr{X}^{(0)} \\ gH & \mapsto & gHg^{-1} \end{array} \right.$$
induces a $C_n \mathcal{G}$-equivariant isomorphism $\Phi: X \to \mathscr{X}$. 
\end{prop}

\noindent
We begin by stating and proving a few preliminary results. First of all, we need to understand the subgroups involved in the definition of $\mathscr{X}$.

\begin{lem}\label{lem:Csubgroups}
The following statements hold:
\begin{itemize}
	\item A subgroup is $\mathcal{C}$-minimal if and only if it is conjugated to $G_i$ for some $i \in \mathbb{Z}/n\mathbb{Z}$.
	\item A subgroup is $\mathcal{C}$-medium if and only if it is conjugated to $\langle G_i, G_{i+1} \rangle$ for some $i \in \mathbb{Z}/n\mathbb{Z}$. 
	\item A subgroup is $\mathcal{C}$-maximum if and only if it is conjugated to $\langle G_{i-1},G_i,G_{i+1} \rangle$ for some $i \in \mathbb{Z}/n\mathbb{Z}$.
\end{itemize}
\end{lem}

\begin{proof}
The $\mathcal{C}$-maximum subgroups are the subgroups of $\mathcal{M}$, which are precisely the conjugates of the $\langle G_{i-1},G_i,G_{i+1} \rangle$'s as a consequence of Corollary \ref{join}. 
%Alternatively, according to ???, they are the wall-stabilisers. 
Therefore, it follows from Proposition \ref{prop:interstabwall} that $\mathcal{C}$ is the collection of the following subgroups:
\begin{itemize}
	\item the $G_i$'s where $i \in \mathbb{Z}/n\mathbb{Z}$;
	\item the $\langle G_i,G_{i+1} \rangle$'s where $i \in \mathbb{Z}/n\mathbb{Z}$;
	\item the $\langle G_{i-1},G_i,G_{i+1} \rangle$'s where $i \in \mathbb{Z}/n\mathbb{Z}$.
\end{itemize}
The desired conclusion follows. 
\end{proof}

\begin{lem}\label{lem:vertexstabwall}
Let $v \in X$ be a vertex and $T \subset X$ a tree-wall. Then $\mathrm{stab}(v)$ stabilises $T$ if and only if $v$ belongs to $T$. 
\end{lem}

\begin{proof}
Suppose that $\mathrm{stab}(v)$ stabilises $T$. Because $T$ is a convex subspace of $X$, endowed with its CAT(0) metric, $\mathrm{stab}(v)$ must  also fix its unique projection $w$ onto $T$, and a fortiori the unique  geodesic $[v, w]$ between $v$ and $w$. Notice that $[v, w]$ cannot intersect the interior of some polygon $Q$ of $X$, since otherwise we would deduce that
$$\mathrm{stab}(v) \subset \mathrm{stab}(Q) = \{ 1 \}.$$
Therefore, $[v, w]$ is contained into the $1$-skeleton of $X$. Similarly, if $v \neq w$ and if $e$ denotes the first edge of $[v,w]$, it follows that $\mathrm{stab}(v) = \mathrm{stab}(e)$, which is impossible since the inclusion $\mathrm{stab}(v) \subset \mathrm{stab}(e)$ is, up to conjugation, of the form $G_i \subset G_i \times G_{i+1}$. Consequently, $v=w \in T$.

\medskip \noindent
Conversely, suppose that $v \in T$. Because the action $C_n \mathcal{G} \curvearrowright X$ preserves the labelling of the edges of $X$, it follows from the definition of a tree-wall that $\mathrm{stab}(v)$ stabilises $T$.
\end{proof}

\begin{lem}\label{lem:vertadjinwall}
Let $T \subset X$ be a tree-wall and $x,y \in T$ two vertices. Then $x$ and $y$ are adjacent if and only if $\langle \mathrm{stab}(x), \mathrm{stab}(y) \rangle = \mathrm{stab}(T)$. 
\end{lem}

\begin{proof}
According to the previous lemma, $H:= \langle \mathrm{stab}(x), \mathrm{stab}(y) \rangle$ stabilises $T$ since $x,y \in T$. 

\medskip \noindent
We claim that, if we set $d=d(x,y)$, then $d(x,gx) \in d \mathbb{Z}$ and $d(y,gy) \in d \mathbb{Z}$ for every $g \in H$. Let us argue by induction on the length of $g$ over the generating set $\mathrm{stab}(x) \cup \mathrm{stab}(y)$. If $g$ has length zero, there is nothing to prove. Now, suppose that $g$ has positive length, and write this element as a word $g_1 \cdots g_r$ of minimal length where $r \geq 1$ and $g_1, \ldots, g_r \in \mathrm{stab}(a) \cup \mathrm{stab}(b)$. For convenience, set $g'=g_2 \cdots g_r$; and suppose that $g_1 \in \mathrm{stab}(x)$ (the case $g_1 \in \mathrm{stab}(y)$ being symmetric). As a consequence of our induction hypothesis, one knows that $d(x,g'x)$ and $d(y,g'y)$ both belong to $d \mathbb{Z}$, which implies that $x$ and $y$ cannot belong to the interior of $[g'x,g'y]$. Because $d(g'x,g'y)=d(x,y)$, it follows that either $x=g'x$ and $y=g'y$, and we are done, or the interiors of $[x,y]$ and $[g'x,g'y]$ are disjoint. We distinguish two cases. First, suppose that $y$ separates $x$ from $[g'x,g'y]$. If $[x,g'y] \cap [x,g_1g'y]$ contains an edge, then $g_1$ must fix it, which implies that $g_1 \in \mathrm{fix}(T)$ according to Lemma \ref{lem:wallfix}, and finally one gets 
$$d(y,gy)= d(g_1y,g_1g'y)=d(y,g'y) \in d \mathbb{Z}.$$
So suppose that $[x,g'y] \cap [x,g_1g'y]$ does not contain any edge, ie., is reduced to $\{x\}$. One has
$$d(y,gy)= d(y,x)+d(x,g_1y)+d(g_1y,g_1g'y) = 2d + d(y,g'y) \in d \mathbb{Z}.$$
Second, suppose that $x$ separates $y$ from $[g'x,g'y]$. As before, if $[x,y] \cap [x,g_1y]$ or $[x,g'y] \cap [x,g_1g'y]$ contains an edge, we deduce that $g_1$ fixes $T$ and next that $d(y,gy) \in d \mathbb{Z}$. So suppose that $[x,y] \cap [x,g_1y]=\{x\}$ and $[x,g'y] \cap [x,g_1g'y]=\{x\}$. Notice that we also have $[x,g_1y] \cap [x,g_1g'y]= \{ x\}$. Indeed, otherwise the first edges of $[x,g_1y]$ and $[x,g_1g'y]$ would coincide, which is impossible since the first edge of $[x,y]$ is sent by $g_1$ to the first edge of $[x,g_1y]$, and the first edge of $[x,g'y]$ is sent by $g_1$ to the first edge of $[x,g_1g'y]$. Consequently,
$$d(y,gy)=d(y,x)+d(x,gy)= d(g_1y,g_1x)+d(x,gy)= d(g_1y,g_1g'y)=d(y,g'y) \in d \mathbb{Z}.$$
Thus, we have proved that $d(y,gy) \in d \mathbb{Z}$. Notice also that
$$d(x,gx)= d(g_1x,g_1g'x)=d(x,g'x) \in d \mathbb{Z}.$$
This concludes the proof of our claim. 

\medskip \noindent
Now, suppose that $x$ and $y$ are not adjacent. As a consequence, there exists a vertex $z \in [x,y]$ which is different from both $x$ and $y$. Suppose that $d(x,z) \leq d(z,y)$ (the case $d(z,y) \leq d(x,z)$ being similar). Similarly to $x$ and $y$, $\mathrm{stab}(z)$ stabilises $T$; moreover, by definition of the polygon of groups defining our polygonal complex $X$, $\mathrm{stab}(z)$ also acts transitively on the edges admitting $z$ as an endpoint. Consequently, there exists some $g \in \mathrm{stab}(z)$ such that $g \cdot x \in [z,y]$. A fortiori, one has $0< d(x,gx) < d$. It follows from our previous claim that necessarily $g \notin H$. Thus, we have proved that $H \subsetneq \mathrm{stab}(T)$.

\medskip \noindent
Conversely, suppose that $x$ and $y$ are adjacent. By construction of $X$, there exists $i \in \bbZ_n$, and an element $g \in C_n \mathcal{G}$ so that $x=g \langle G_{i-1}, G_i \rangle$ and $y=g \langle G_i,G_{i+1} \rangle$. Consequently,
$$\langle \mathrm{stab}(x), \mathrm{stab}(y) \rangle = g \cdot \langle G_{i-1},G_i,G_{i+1} \rangle \cdot g^{-1},$$
which coincides with the $\mathrm{stab}(T)$ according to Lemma \ref{lem:stab_treewall}.
\end{proof}

\begin{lem}\label{lem:cyclesfill}
Every induced cycle of length $n$ in $X$ bounds a polygon.
\end{lem}

\begin{proof}
Let $\gamma$ be a induced cycle of length $n$ in $X$. Fix a reduced disc diagram $D \to X$ bounding $\gamma$. Our goal is to show that $D$ is necessarily a polygon, which implies the desired conclusion. Since $X$ is a $T(4)$ complex, the only vertices of $D$ contributing to positive curvature are (possibly) vertices of $\partial D$, which contribute to at most $\frac{\pi}{2}$ each. Each polygon of $D$ contributes to $2\pi - n \frac{\pi}{2}$, which is negative as $n \geq 5$. Thus, we have 
$$\underset{v}{\sum} \kappa_D(v) + \underset{f}{\sum} \kappa_D(f)  \leq n \frac{\pi}{2} + \left(2\pi - n \frac{\pi}{2} \right) = 2\pi. $$
The combinatorial Gau\ss--Bonnet Theorem now implies that the previous inequality must be an equality. In particular, there can only be one polygon in $D$, which concludes the proof.
%According to Theorem ???, two cases must be excluded: $D$ is a ladder which is not reduced to a single polygon, and $D$ contains at least three shells. 
%
%\medskip \noindent
%First, suppose that $D$ is a ladder which is not reduced to a single polygon. Let $R_1,R_2$ denote the two extremal polygons of $D$. Notice that, because the intersection of any two polygons contains at most one edge, one has
%$$|\partial D| \geq |\partial R_1 \cap \partial D|+ |\partial R_2 \cap \partial D| \geq |\partial R_1|-1+ |\partial R_2|-1 = 2n-2>n,$$
%because $n \geq 5$, which is impossible since $|\partial D |= |\gamma|=n$. Next, suppose that $D$ has at least three shells $R_1,R_2,R_3$. Similarly, one deduces that
%$$|\partial D| \geq |\partial R_1 \cap \partial D| +  |\partial R_2 \cap \partial D|  +  |\partial R_3 \cap \partial D| \geq |R_1| - 3 + |R_2|-3+|R_3|-3 = 3n-9>n,$$
%which is contradictory. This concludes the proof. 
\end{proof}

\begin{proof}[Proof of Proposition \ref{prop:algch}.]
The fact that the map
$$\left\{ \begin{array}{ccc} X^{(0)} & \to & \mathscr{X}^{(0)} \\ gH & \mapsto & gHg^{-1} \end{array} \right.$$
is a bijection follows from Lemma \ref{normalizer}, which shows that a subgroup $\langle G_i,G_{i+1} \rangle$, where $i \in \mathbb{Z}/n\mathbb{Z}$, is self-normalising, and from Lemma \ref{lem:Csubgroups}, which describes the vertices of $\mathscr{X}$. Next, let $g_1H_1$ and $g_2H_2$ be two vertices of $X$. If they are adjacent, then it follows from Lemma \ref{lem:vertadjinwall} that $\langle g_1H_1g_1^{-1},g_2H_2g_2^{-1} \rangle$ coincides with the stabiliser of the tree-wall passing through $g_1H_1$ and $g_2H_2$, and so is a $\mathcal{C}$-maximal subgroup. A fortiori, $g_1H_1g_1^{-1}$ and $g_2H_2g_2^{-1}$ are two adjacent vertices of $\mathscr{X}$. Conversely, suppose that $g_1H_1g_1^{-1}$ and $g_2H_2g_2^{-1}$ are two adjacent vertices of $\mathscr{X}$. It means that $\langle g_1H_1g_1^{-1},g_2H_2g_2^{-1} \rangle$ coincides with the stabiliser of some tree-wall $T \subset X$. As a consequence of Lemma \ref{lem:vertexstabwall}, necessarily the vertices $g_1H_1$ and $g_2H_2$ both belongs to $T$, and finally we conclude from Lemma \ref{lem:vertadjinwall} that $g_1H_1$ and $g_2H_2$ must be adjacent in $X$. Thus, we have proved that our map induces an isomorphism $X^{(1)} \to \mathscr{X}^{(1)}$. Finally, the conclusion follows from Lemma \ref{lem:cyclesfill}.
\end{proof}

%% file: Automorphisms.tex
\section{Automorphisms}\label{section:aut}

The previous algebraic characterisation of the Davis complex allows us to extend the action of $C_n\cG$ to an action of $\mathrm{Aut}(C_n\cG)$, as we explain in this section.

\begin{prop}
The automorphism group $\mathrm{Aut}(C_n \mathcal{G})$ acts by isometries on $\mathscr{X}$ via
$$\varphi \cdot H = \varphi(H) \ \text{for every $\varphi \in \mathrm{Aut}(C_n \mathcal{G})$ and every $H \in \mathscr{X}^{(0)}$}.$$
Moreover, the map $\Phi$ of Proposition \ref{prop:algch} satisfies
$$\Phi(g \cdot x) = \iota(g) \cdot \Phi(x)$$
for every $g \in C_n \mathcal{G}$ and every $x \in X^{(0)}$, where $\iota(g)$ denotes the conjugacy by $g$.
\end{prop}

\begin{proof}
Since that the canonical action $\mathrm{Aut}(C_n \mathcal{G}) \curvearrowright C_n \mathcal{G}$ preserves the collection of subgroups $\mathcal{M}$, it follows from the construction of $\mathscr{X}$ that $\mathrm{Aut}(C_n \mathcal{G})$ acts on $\mathscr{X}$ by isometries. Next, fix an element $g \in C_n \mathcal{G}$ and a vertex $x \in X^{(0)}$. Write $x=h \langle G_i,G_{i+1} \rangle$ for some element $h \in C_n \mathcal{G}$ and some index $i \in \mathbb{Z}/n\mathbb{Z}$; for short, set $H= \langle G_i,G_{i+1} \rangle$. Then
$$\Phi(g \cdot x) = \Phi (ghH) = ghHh^{-1}g^{-1} = \iota_g \left( hHh^{-1} \right) = \iota_g \cdot \Phi(hH) = \iota_g \cdot \Phi(x).$$
This concludes the proof. 
\end{proof}

\noindent
By transfering the action $\mathrm{Aut}(C_n \mathcal{G}) \curvearrowright \mathscr{X}$ via the map given by Proposition \ref{prop:algch}, one immediately gets:

\begin{cor}
The automorphism group $\mathrm{Aut}(C_n \mathcal{G})$ acts by isometries on the Davis complex $X$ via
$$\varphi \cdot kH = \varphi(k) \varphi(H) \ \text{for every $\varphi \in \mathrm{Aut}(C_n \mathcal{G})$ and every $kH \in X^{(0)}$.}$$
Moreover, this action extends $C_n \mathcal{G} \curvearrowright X$ if one identifies canonically $C_n \mathcal{G}$ with $\mathrm{Inn}(C_n \mathcal{G})$. 
\end{cor}

We now use the geometry of the action $\mathrm{Aut}(C_n \mathcal{G}) \curvearrowright X$ to completely compute the automorphism group $\mathrm{Aut}(C_n \mathcal{G}).$

\begin{definition}[local automorphism]
Let $(\sigma, \Phi)$ be the data of a symmetry $\sigma$ of $C_n$, whose vertices are identified with the elements of $\mathbb{Z}/n\mathbb{Z}$, and a collection of isomorphisms $\Phi = \{ \varphi_i : G_i \to G_{\sigma(i)} \mid i \in \mathbb{Z}/n\mathbb{Z} \}$. Such a couple $(\sigma, \Phi)$ naturally defines an automorphism of $C_n \mathcal{G}$ by $g \mapsto \varphi_i(g)$ for every $i \in \mathbb{Z}/n\mathbb{Z}$ and every $g \in G_i$, which will be referred to as a \emph{local automorphism} of $C_n \mathcal{G}$. We denote by $\mathrm{Loc}(C_n \mathcal{G})$ the subgroup of the local automorphisms and by $\mathrm{Loc}^0(\Gamma \mathcal{G})$ the subgroup of the local automorphisms associated to pairs $(\sigma, \Phi)$ where $\sigma= \mathrm{Id}$.
\end{definition}

\noindent
It is worth noticing that $\mathrm{Loc}(C_n \mathcal{G})$ decomposes as semi-direct product $\mathrm{Loc}^0(C_n \mathcal{G}) \rtimes \mathrm{Sym}$, where $\mathrm{Sym}$ is the subgroup of the symmetry group of $C_n$ preserving the isomorphism classes of the vertex-groups. However, in general there is no canonical isomorphism between these groups. For convenience, from now on we will denote $\mathrm{Loc}(C_n \mathcal{G})$ by $\mathrm{Loc}$, and similarly the subgroup $\mathrm{Inn}(C_n \mathcal{G})$ of the inner automorphisms will be denoted by $\mathrm{Inn}$.

Before computing $\mathrm{Aut}(C_n \mathcal{G})$, we mention to elementary results about the action $\mathrm{Aut}(C_n \mathcal{G}) \curvearrowright X$. Since there is a single $C_n\cG$-orbit of polygons of $X$, Proposition \ref{prop:algch} implies that there is a single orbit of $\mathrm{Aut}(C_n \mathcal{G}) $-orbit of polygons of $\mathscr{X}$.

\begin{definition}[fundamental polygon]\label{def:fundpoly}
	We denote by $P$ the polygon of $\mathscr{X}$ whose vertices are 
	$$\langle G_1,G_2 \rangle, \ldots, \langle G_{n-1},G_n \rangle, \langle G_n,G_1 \rangle.$$
\end{definition}

\begin{lem}\label{fact_1}
	The $\mathrm{Aut}(C_n \mathcal{G})$-stabiliser of $P$ is $\mathrm{Loc}$.
\end{lem}

\begin{proof} 
	Let $\psi$ be an automorphism stabilising $P$. Because $\psi$ permutes the vertices of $P$, there exists a bijection $\sigma : \mathbb{Z}/n \mathbb{Z} \to \mathbb{Z} /n \mathbb{Z}$ such that $\psi \left( \langle G_i,G_{i+1} \rangle \right) = \langle G_{\sigma(i)}, G_{\sigma(i+1)} \rangle$ where $|\sigma(i)-\sigma(i+1)|=1$ for every $i$ mod $n$. For every index $i$ mod $n$, one has
	$$\psi(G_i)= \psi \left( \langle G_{i-1},G_i \rangle \cap \langle G_i,G_{i+1} \rangle \right) = \langle G_{\sigma(i-1)}, G_{\sigma(i)} \rangle \cap \langle G_{\sigma(i)}, G_{\sigma(i+1)} \rangle = G_{\sigma(i)},$$
	the last equality being justified by the observation that, because $\psi$ is an automorphism, necessarily $\psi( \langle G_{i-1},G_i \rangle) \neq \psi( \langle G_i,G_{i+1} \rangle)$, which implies that $\sigma(i-1) \neq \sigma(i+1)$. Consequently, if we set $\Phi = \{ \psi_{|G_i} : G_i \to G_{\sigma(i)} \mid i \ \text{mod} \ n \}$, then
	$$\psi = (\sigma, \Phi) \in \mathrm{Loc}.$$
	Thus, we have proved that our stabiliser is included into $\mathrm{Loc}$. The converse is clear, which concludes the proof.
\end{proof}

\begin{lem}\label{fact_2}
	$\mathrm{Inn} \cap \mathrm{Loc}= \{ \mathrm{Id} \}$.
\end{lem}

\begin{proof}
	Because the $C_n \mathcal{G}$-stabiliser of a polygon of $X$ is trivial, we deduce from Proposition \ref{prop:algch} that the $\mathrm{Inn}$-stabiliser of a polygon of $\mathscr{X}$ must be trivial as well. The desired conclusion follows from Lemma \ref{fact_1}. 
\end{proof}

We are now ready to prove the main result of this section:

\begin{thm}\label{thm:auto}
For every $n \geq 5$,
$$\mathrm{Aut}(C_n\mathcal{G})= \mathrm{Inn} \cdot \mathrm{Loc} \simeq C_n\mathcal{G} \rtimes \left( \left( \prod\limits_{i\in\bbZ_n} \mathrm{Aut}(G_i) \right) \rtimes \mathrm{Sym} \right).$$
\end{thm}

\begin{proof}
%Let $v_1, \ldots, v_n$ denote the vertices of $C_n$, indexed cyclically. To avoid double subscripts, set $G_i=G_{v_i}$ for every $1 \leq i \leq n$. It follows from Lemma ??? and from the construction of $\mathscr{X}$ that 
%$$\langle G_1,G_2 \rangle, \ldots, \langle G_{n-1},G_n \rangle, \langle G_n,G_1 \rangle$$
%defines an $n$-cycle in $\mathscr{X}$. According to Lemma ???, such a cycle must bound a polygon, which we denote by $P$. Now, l
Let $\varphi \in \mathrm{Aut}(C_n \mathcal{G})$ be an automorphism. 
Since $X$ contains a unique $C_n \mathcal{G}$-orbit of polygons, it follows from Proposition \ref{prop:algch} that $\mathscr{X}$ contains a unique $\mathrm{Inn}$-orbit of polygons. Therefore, there exists some $\iota \in \mathrm{Inn}$ such that $\iota \varphi \cdot P = P$, where $P$ denotes the fundamental polygon. Equivalently, there exists some $\iota \in \mathrm{Inn}$ such that $\iota \varphi$ belongs to the $\mathrm{Aut}(C_n \mathcal{G})$-stabiliser of $P$. Lemma \ref{fact_1}  implies that $\varphi \in \mathrm{Inn} \cdot \mathrm{Loc}$, which proves the equality $\mathrm{Aut}(C_n \mathcal{G}) = \mathrm{Inn} \cdot \mathrm{Loc}$. Finally, the equality $\mathrm{Aut}(C_n \mathcal{G}) = \mathrm{Inn} \rtimes \mathrm{Loc}$ follows from Lemma \ref{fact_2}.
\end{proof}

%% file: Acyl.tex
\section{Acylindrical hyperbolicity}

The aim of this section is to prove the following:

\begin{thm}\label{thm:acyl_hyp}
Assume that each group $G_i$, $1\leq i \leq n$, has its automorphisms determined by a finite set. Then Aut$(C_n\cG)$ is acylindrically hyperbolic.
\end{thm}

%\noindent
%Given an arbitrary group $G$ and a subset $S \subset G$, say that $S$ is a \emph{fixing subset} if the only automorphism of $G$ fixing $S$ pointwise is the identity. Typically, generating subsets are the examples we have in mind. 

%\begin{thm}
%For every $n \geq 5$ and every collection of nontrivial groups $\mathcal{G}$ such that each group admits a finite fixing subset, the automorphism group $\mathrm{Aut}(C_n \mathcal{G})$ is acylindrically hyperbolic.
%\end{thm}

To prove this result, we use a criterion introduced in \cite{Chatterji_Martin_Criterion_Acyl_Hyp} to show the acylindrical hyperbolicity of a group via its action on a CAT(0) cube complex. Following \cite{CapraceSageev}, we say that the action of a group $G$ on a CAT(0) cube complex $Y$ is \textit{essential} if no $G$-orbits stays in some neighbourhood of a half-space. Following \cite{ChatterjiFernosIozzi}, we say that the action is \textit{non-elementary} if $G$ does not have a fixed point in $Y \cup \partial_{\infty}Y$.  We further say that $Y$ is \textit{cocompact } if its automorphism group acts cocompactly on it; that it is \textit{irreducible} if it does not split as the direct product of two non-trivial CAT(0) cube complexes; and that it \textit{does not have a free face} if every non-maximal cube is contained in at least two maximal cubes. 

\begin{thm}[{\cite[Theorem 1.5]{Chatterji_Martin_Criterion_Acyl_Hyp}}]\label{thm:Chatterji_Martin}
Let $G$ be a group acting  non-elementarily and essentially on a  finite-dimensional irreducible cocompact CAT(0) cube complex with no free face. If there exist two points whose stabilisers intersect along a finite subgroup, then $G$ is acylindrically hyperbolic.
\end{thm}

We will use this criterion for the action of Aut$(C_n\cG)$ on the Davis complex $X'$. To this end, we need to check a few preliminary results about the action. 

%\begin{prop}
%The action of Aut$(C_n\cG)$ on $X$ is essential and non-elementary.
%\end{prop}
%
%We split the proof of this result into several steps: 

\begin{lem}\label{lem:ess}
The action of Aut$(C_n\cG)$ on $X'$ is essential and non-elementary.
\end{lem}

\begin{proof}
It is enough to show that the action of $C_n\cG$ (identified with the subgroup of Aut$(C_n\cG)$ consisting of inner automorphisms) acts essentially and non-elementarily on $X'$. 

\textbf{Non-elementarity.} Since $X'$ is hyperbolic, non-elementarity of the action will follow from the fact that there exist two elements $g, h \in C_n\cG$ acting hyperbolically on $X'$ and having disjoint limit sets in the Gromov boundary of $X'$, by elementary considerations of the dynamics of the action on the boundary of a hyperbolic space. We now construct such hyperbolic elements.

Let $P'$ be the fundamental domain of $X'$ for the action of $C_n\cG$, ie. the subdivision of the polygon of $X$ corresponding to the polygon of $\mathscr{X}$ given by Definition \ref{def:fundpoly}. For each $1 \leq i \leq n$, let $e_i$ be the edge of $P'$ whose stabiliser is $G_i$, and let $T_i$ be the associated tree-wall. For each $1 \leq i \leq n$, choose a non-trivial element $s_i \in G_i$ and define the group element $g_i:= s_{i-1}s_{i+1} \in C_n\cG$, where the indices are considered modulo $n$.  Let $e_{i}' := s_{i+1}e_{i}$  and define 
$$\Lambda_{i} := \bigcup_{k\in \bbZ} g_i^k (e_i \cup e_i').$$
Then $\Lambda_i$ is a combinatorial axis for $g_i$, contained in the tree-wall $T_i$: Indeed, notice that $\Lambda_i$ makes an angle $\pi$ at each vertex.
We now claim that the limit sets of two consecutive such axes $\Lambda_i$, $\Lambda_{i+1}$ are disjoint. To show this, it is enough to show the analogous result for the limit sets of $T_i$ and $T_{i+1}$, and in particular it is enough to show that for every vertex $v$ of $T_i$, its unique projection on $T_{i+1}$ for the combinatorial metric is exactly their intersection point $u:= T_i \cap T_{i+1}$. Suppose by contradiction that this is not the case for some vertex $v \in T_i$. Then there exists an edge of $T_i$ between $u$ and $v$ that defines a hyperplane  that crosses $T_{i+1}$. Let $T$ be the tree-wall of $X$ associated to that hyperplane. We thus have three pairwise intersecting tree-walls, which implies that there exist three pariwise intersecting hyperplanes of $X'$. But this is impossible since $X'$ is $2$-dimensional. 
%\textcolor{red}{(Alternative and shorter argument to conclude: if there exist three tree-walls pairwise intersecting, then there exist three hyperplanes pairwise intersecting in $X'$, which is impossible since $X'$ is two-dimensional.)} Then $W_i$, $W_{i+1}$, and $W$ define a loop in $X$, and we now choose a reduced disc diagram bounding that loop. Note that at most three vertices on the boundary of that loop contribute to positive curvature, since adjacent edges of a given tree-wall make an angle $\pi$. Moreover, each internal vertex contributes to non-positive curvature since $X$ is a CAT(0) square complex. This in turn violates the Gau\ss-Bonnet Theorem, yielding our contradiction.

\textbf{Essentiality.}  Let $h$ be a halfspace of $X'$ associated to a hyperplane $\hat{h}$. Up to the action of $C_n\cG$, we can assume that  the tree-wall associated to $\hat{h}$ is  $T_i$ for some $1\leq i \leq n$.  We showed in the above paragraph that for each $1 \leq i \leq n$, there exist points in the combinatorial axis $\Lambda_{i+1}$ at distance as large as desired from $T_{i}$. This immediately translates into the existence of points in $\Lambda_{i+1}$  at distance as large as desired from $h$, which concludes.
\end{proof}

%\begin{lem}
%The action of Aut$(C_n\cG)$ on $X$ is non-elementary.
%\end{lem}
%
%\begin{proof}
%
%\end{proof}

\begin{lem}\label{lem:irred}
The Davis complex $X'$ is irreducible.
\end{lem}

\begin{proof}
The link of every vertex of $X'$ corresponding to a coset of the trivial subgroup has a link which is a cycle on $n$ vertices. As $n\geq 5$, such a link does not decompose non-trivially as a join, hence $X'$ does not decompose non-trivially as a direct product.
\end{proof}

\begin{lem}\label{lem:no_free_face}
The Davis complex $X'$ has no free face. 
\end{lem}

\begin{proof}
It is enough to prove that every edge of $X'$ contained in $X^{(1)}$ is contained in at least two squares. Let $e$ be an edge of $X$ and let $C$ be a square containing $e$. There are exactly $[\mathrm{Stab}_{C_n\cG}(e) : \mathrm{Stab}_{C_n\cG}(C)]$  $C_n\cG$-translates of $C$ containing $e$. As $\mathrm{Stab}_{C_n\cG}(C)$ is trivial and $\mathrm{Stab}_{C_n\cG}(e)$, which is conjugate to some $G_i$, contains at least two elements, it follows that there are at least two squares containing $e$. 
\end{proof}

\begin{lem}\label{lem:stab}
Let $P$ be the fundamental domain of $X$ and let $g \in C_n\cG$. Then 
$$ \mbox{Stab}_{\mathrm{Aut}(C_n\cG)}(P)\cap  \mbox{Stab}_{\mathrm{Aut}(C_n\cG)}(gP) = \{ \varphi \in \mbox{Loc}(C_n\cG) \mid \varphi(g)=g\}.$$	
\end{lem}

\begin{proof}
Recall that $\mbox{Stab}_{\mathrm{Aut}(C_n\cG)}(P) = \mathrm{Loc}(C_n \cG)$. Therefore, if $\varphi \in \mathrm{Aut}(C_n \cG)$ belongs to $\mbox{Stab}_{\mathrm{Aut}(C_n\cG)}(P)\cap  \mbox{Stab}_{\mathrm{Aut}(C_n\cG)}(gP)$ then $\varphi \in  \mathrm{Loc}(C_n \cG)$ and there exists some $\psi \in \mathrm{Loc}(C_n \cG)$ such that $\varphi = \iota(g) \circ \psi \circ \iota(g)^{-1}$, where $\iota(g)$ denotes the inner automorphism defined by $g$. Since $\psi \circ \iota(g)^{-1} = \iota(\psi(g))^{-1} \circ \psi$, we deduce that
$$\varphi \circ \psi^{-1}= \iota(g) \circ \psi \circ \iota(g)^{-1} \circ \psi^{-1} = \iota(g) \circ \iota(\psi(g))^{-1} \circ \psi \circ \psi^{-1},$$
hence
$$\varphi \circ \psi^{-1} = \iota(g) \circ \iota(\psi(g))^{-1} \in \mathrm{Inn}(C_n \mathcal{G}) \cap \mathrm{Loc}(C_n \cG).$$
On the other hand, we know from Lemma \ref{fact_2} that $\mathrm{Inn}(C_n \mathcal{G}) \cap \mathrm{Loc}(C_n \cG)= \{ \mathrm{Id} \}$, whence $\varphi= \psi$ and $\iota(g)= \iota(\psi(g))$. As $C_n\cG$ is centerless by Corollary \ref{center}, this implies $\varphi(g)=g$, hence  the inclusion 
$$\mbox{Stab}_{\mathrm{Aut}(C_n\cG)}(P)\cap  \mbox{Stab}_{\mathrm{Aut}(C_n\cG)}(gP) \subset \{ \varphi \in \mbox{Loc}(C_n\cG) \mid \varphi(g)=g\}.$$
The reverse inclusion is clear.
\end{proof}

\begin{proof}[Proof of Theorem \ref{thm:acyl_hyp}.]
For each $1 \leq i\leq n$, choose a finite family $\{s_{i,j} \mid 1 \leq j \leq m_i \}$ determining the automorphisms of $G_i$. Up to allowing repetitions, we will assume that all the integers $m_i$ are equal, and denote by $m$ that integer. We now define a specific element $g \in C_n \mathcal{G}$ in the following way:
$$g_{i,j}:= s_{i+2,j}s_{i,j} \mbox{ for $1 \leq i \leq n$, $1 \leq j \leq m$,}$$
$$g_{j}:= g_{1,j} \cdots g_{n,j} \mbox{ for $1 \leq j \leq m$,}$$
$$g:=g_1 \cdots g_m.$$
Let $\varphi$ be an element of $ \mbox{Stab}_{\mathrm{Aut}(C_n\cG)}(P)\cap  \mbox{Stab}_{\mathrm{Aut}(C_n\cG)}(gP)$. By Lemma \ref{lem:stab}, it follows that $\varphi \in \mathrm{Loc}(C_n\cG)$ and $\varphi(g) = g$. By construction, $g$ can be written as a concatenation of the form $g= s_1 \cdots s_p$, where each $s_k$ is of the form $s_{i,j}$, and such that no consecutive $s_k, s_{k+1}$ belong to groups of $\cG$ that are joined by an edge of $\Gamma$. In particular, the decomposition $g= s_1 \cdots s_p$ is the unique reduced form of $g$ by Corollary \ref{cor:uniquereduced}. As $g = \varphi(g) = \varphi(s_1)\cdots\varphi(s_p)$ is an another reduced form of $g$, it follows that $\varphi(s_k)=s_k$ for every $k$. As we have 
$$\{s_k, 1 \leq k \leq p\} = \{s_{i,j}, 1\leq i \leq n, 1\leq j \leq m\}$$
by construction of $g$, it follows from the construction of the elements $s_{i,j}$ that $\varphi$ induces the identity automorphism on each $G_i$, hence $\varphi$ is the identity. We thus have that $ \mbox{Stab}_{\mathrm{Aut}(C_n\cG)}(P)\cap  \mbox{Stab}_{\mathrm{Aut}(C_n\cG)}(gP)$ is trivial. It now follows from Lemmas  \ref{lem:ess}, \ref{lem:irred}, \ref{lem:no_free_face}, and \ref{lem:stab} that Theorem \ref{thm:Chatterji_Martin} applies, hence Aut$(C_n\cG)$ is acylindrically hyperbolic.
%
%Thus,  $ \mbox{Stab}_{\mathrm{Aut}(C_n\cG)}(Y_{-1}) \cap \mbox{Stab}_{\mathrm{Aut}(C_n\cG)}(Y_{k+1})$ is trivial, and the result now follows from Theorem \ref{thm:criterionhypacyl} and Lemma \ref{lem:ess}.
\end{proof}

\noindent
Let us conclude this section by mentioning an example of cyclic product whose automorphism group is not acylindrically hyperbolic.

\begin{rmk}\label{remark:nonacyl}
Let $Z$ be the direct sum $\bigoplus\limits_{p \ \text{prime}} \mathbb{Z}_p$ and let $G_n$ be the graph product of $n$ copies of $Z$ over the cycle $C_n$. Let $g \varphi \in \mathrm{Aut}^0(G_n)$ where $g \in \mathrm{Inn}$ and $\varphi \in \mathrm{Loc}$. For each copy $Z_i$ of $Z$, the reduced word representing $g$ contains only finitely-many syllables in $Z_i$; let $S_i \subset Z_i$ denote this set of syllables. Clearly, there exists an infinite collection of automorphisms of $Z_i$ fixing $S_i$ pointwise; furthermore, we may suppose that this collection generates a subgroup of automorphisms $\Phi_i \leq \mathrm{Aut}(Z_i)$ which is a free abelian group of infinite rank. Notice that $\phi(g)=g$ for every $\phi \in \Phi_i$. Therefore, for every $\psi \in \Phi_1 \times \cdots \times \Phi_n \leq \mathrm{Loc}$, we have
	\begin{center}
		$\psi \cdot g \varphi= \psi(g) \cdot \psi \varphi=g \cdot \psi \varphi= g \cdot \varphi \psi = g \varphi \cdot \psi$,
	\end{center}
	since $\varphi$ and $\psi$ clearly commute: each $\mathrm{Aut}(Z_i)$ is abelian so that $\mathrm{Loc}$ is abelian as well. Thus, we have proved that the centraliser of any element of $\mathrm{Aut}^0(G_n)$ contains a free abelian group of infinite rank. Therefore, $\mathrm{Aut}^0(G_n)$ (and a fortiori $\mathrm{Aut}(G_n)$) cannot be acylindrically hyperbolic.
\end{rmk}

%Thingsthat we need in previous sections:
%- uniqueness of reduced form for certain words,
%- wall associated to a hyperplane,
%- Gauss-Bonnet (polygonal and squarical)
%- C_nG centerless